\documentclass[11pt,letterpaper]{amsart}

\usepackage{tikz}
\usetikzlibrary{arrows.meta}
\usepackage{url}
\usepackage{graphicx}

\usepackage{amssymb}
\usepackage[OT1]{fontenc}
\usepackage{enumitem}

\allowdisplaybreaks%

{\theoremstyle{plain}%
 \newtheorem{theorem}{Theorem}
 
 \newtheorem{lemma}{Lemma}%
}
{\theoremstyle{remark}

}
{\theoremstyle{definition}

}

\title[Pinch's conjecture on $a$-convexity]{Pinch's conjecture on $a$-convexity}
\author{John M.\ Campbell}
\address{Department of Mathematics and Statistics, Dalhousie University,
Halifax, NS B3H 4R2, Canada}
\email{jh241966@dal.ca}
 
\keywords{convex hull, algebraic integer, totally real number, algebraic conjugate, Bernstein basis polynomial}

\subjclass[2020]{Primary 11R04; Secondary 41A10}

\begin{document}

\begin{abstract}
 For $a$ in $\mathbb{R}$, a subset $V$ contained in $\mathbb{R}^{n}$ is said to be \emph{$a$-convex} if $\text{{\bf x}}, \text{{\bf y}} \in V 
 \Longrightarrow a \text{{\bf x}} + (1-a) \text{{\bf y}} \in V$. According to Pinch [\emph{Math.\ Proc.\ Cambridge Philos.\ Soc.}, 1985], the 
 \emph{$a$-convex hull} of $V$ is the intersection of all of the $a$-convex subsets of $\mathbb{R}^{n}$ that contain $V$, and Pinch also 
 defines $D(a)$ as the $a$-convex hull of $\{ 0, 1 \}$ in $\mathbb{R}^{1}$. Pinch conjectured that if $a$ is a totally real algebraic integer 
 and $D(a)$ has no limit points, then every algebraic conjugate of $a$ other than $a$ is in $(0, 1)$. We succeed in proving this 
 conjecture, which seems to have remained open. 
 \end{abstract}

\maketitle

\section{Introduction}
 The purpose of this paper is to prove a conjecture due to Pinch \cite{Pinch1985} related to what Pinch refers to as an \emph{$a$-convex 
 set}, referring to the preliminaries covered below. Based on extant work citing Pinch's paper 
 \cite{BhattacharyaRosenfeld2000,FennerGreenHomer2026,MasakovaPateraPelantova2001,MasakovaPelantovaSvobodova2000,
Svobodova2001},  it appears that Pinch's conjecture has remained open.   Our proof of Pinch's conjecture is based on our extensive 
 interactions with GPT-5.6 Pro. 

 Recall that an \emph{algebraic integer} is a root of a monic polynomial with coefficients in $\mathbb{Z}$. 
 For an algebraic number $\alpha$, letting $m_{\alpha}(X) \in \mathbb{Q}[X]$ denote
 its minimal polynomial over $\mathbb{Q}$, 
 an \emph{algebraic conjugate} of $\alpha$ refers to any complex root of $m_{\alpha}(X)$. 
 An algebraic number is said to be \emph{totally real} if each algebraic conjugate of this number is a real number. 
 A \emph{nontrivial conjugate} of $\alpha$ refers to a conjugate distinct from $\alpha$. 

 Following Pinch's work \cite{Pinch1985}, 
 a set $V \subseteq \mathbb{R}^{n}$ is \emph{$a$-convex} for a value $a \in \mathbb{R}$
 if: For all $\text{{\bf x}}$ and $\text{{\bf y}}$ in $V$, we have that 
 $a \text{{\bf x}} + (1-a) \text{{\bf y}}$ is in $V$. 
 Also, the \emph{$a$-convex hull} of $V$ is defined by Pinch as the intersection 
 of the $a$-convex subsets that contain $V$ and that are contained in $\mathbb{R}^{n}$. 
 Pinch also defines $D(a)$ as the $a$-convex hull of $\{ 0, 1 \}$ in $\mathbb{R}$. 

 For an indeterminate $X$, let $D_0(X) = \{ 0, 1 \}$. 
 We then recursively define $D_{n+1}(X)$ so that 
 $D_{n+1}(X) = \{ X f(X) + (1-X) g(X) : f, g \in D_n(X) \}$. 
 Moreover, we let $D(X) = \bigcup_{n=0}^{\infty} D_n(X)$, and for $a \in \mathbb{R}$, 
 we let 
\begin{equation}\label{defineDa}
 D(a) = \{ f(a) : f \in D(X) \}. 
\end{equation}
 It is proved by Pinch \cite[Proposition 1]{Pinch1985} that $D(a)$, as defined in 
 \eqref{defineDa}, is the $a$-convex hull of $\{ 0, 1 \}$.
 Again adopting notation from Pinch, 
 a real parameter $a$ is said to be \emph{sparse} if $D(a)$ 
 does not have any limit points. 

 Pinch \cite[Proposition 12]{Pinch1985} proved that: If $a \not\in (0, 1)$ is a totally real algebraic integer and if each nontrivial conjugate of 
 $a$ is in $(0, 1)$, then $a$ is sparse. 
 Pinch also proved the converse for the quadratic case, and considered the problem
 as to whether or not the converse holds in full generality. 
 This is made explicit below, via what we refer to as \emph{Pinch's conjecture}, 
 as given in an equivalent way by Pinch in 1985. 

 \ 

\noindent {\bf Pinch's conjecture:} Let $a$ denote a totally real algebraic integer.
 If $a$ is sparse, then each conjugate of $a$ distinct from $a$ 
 is in $(0, 1)$. 

 \ 

\noindent The main purpose of this paper is to prove Pinch's conjecture, 
 as in Section \ref{secmain} below. 

\section{Main result}\label{secmain}
 The following result is required for our purposes and was established by Pinch 
 \cite[Corollary 4.1]{Pinch1985}, and we thus omit a proof of 
 the following result from Pinch. 

\begin{lemma}\label{Pinchlemma}
 (Pinch, 1985) For each $n \geq 0$, 
 the equality 
 $$ D_n(X) = \left\{ \sum_{k=0}^{n} c_k X^k(1-X)^{n-k} 
 : c_k \in \mathbb{Z}, \ 0 \leq c_k \leq \binom{n}{k} \right\} $$
 holds \cite[Corollary 4.1]{Pinch1985}. 
\end{lemma}

 For $0 \leq k \leq n$, we let \emph{Bernstein basis polynomials} be defined and denoted so that 
\begin{equation}\label{displayBernstein} 
 B_{k, n}(X) = \binom{n}{k} X^{k} (1-X)^{n-k}. 
\end{equation}
 For fixed $n$, the family of expressions of the form shown in \eqref{displayBernstein} 
 is a basis of 
 $\{ p(X) \in \mathbb{R}[X] : \deg p \leq n \}$.

\begin{lemma}\label{lemmalargen}
 Let $R \in \mathbb{R}[X]$ satisfy $R(x) > 0$ for $x \in (0, 1)$
 and $R(0), R(1) \geq 0$. 
 Then, for all sufficiently large $n$, the coefficients of $R$ as a linear combination of expressions of the form
 $B_{k, n}(X)$ are all nonnegative. 
\end{lemma}

\begin{proof}
 We begin with the factorization 
\begin{equation}\label{defineRX}
 R(X) = X^r (1-X)^s H(X)
\end{equation}
 for $r, s \geq 0$ and where neither $1-X$ nor $X$ divides $H(X)$, noting that $R(X)$ is not the zero polynomial from the assumption 
 that $R(x) > 0$ for $x \in (0, 1)$. From the assumption that $R$ is positive on $(0, 1)$, this and \eqref{defineRX} together give us that 
 $H(x) > 0$ for $x \in (0, 1)$. 
 The specified values of $r$ and $s$ are, respectively, the exact multiplicities of the zeros of $R$ at $0$ and $1$. 
 This gives us that $H(0) \neq 0$ and that $H(1) \neq 0$. 
 Since $x^r$ and $(1-x)^s$ are strictly positive for $x \in (0, 1)$, 
 and since $H(x) = \frac{R(x)}{x^r (1-x)^s}$ is continuous and since $R(x) > 0$ for $x \in (0, 1)$, 
 we find that $H(0) = \lim_{x \to 0^{+}} H(x) \geq 0$. 
 From this and $H(0) \neq 0$ together, we have that $H(0) > 0$, 
 and a symmetric argument allows us to similarly conclude that $H(1) > 0$. 
 So, we have that $H(x) > 0$ for $x \in [0, 1]$. 
 Now, write 
\begin{equation}\label{coefficientsHX} 
 H(X) = \sum_{j=0}^{\deg H} h_j X^j,
\end{equation}
 for real coefficients 
 $h_{j}$ for $j \in \{ 0, 1, \ldots, \deg H \}$. 
 For an arbitrary integer $m \geq \deg H$, we 
 obtain a unique degree-$m$ Bernstein expansion of $H(X)$, writing
\begin{equation}\label{HgammaB}
 H(X) = \sum_{k=0}^{m} \gamma_{k, m} B_{k, m}(X), 
\end{equation}
 with 
\begin{equation}\label{definegamma}
 \gamma_{k, m} = 
 \sum_{j=0}^{\min\{ \deg H, k \}} h_{j} \frac{ \binom{k}{j} }{ \binom{m}{j} } = 
 \sum_{j=0}^{ \deg H } h_{j} \frac{ \binom{k}{j} }{ \binom{m}{j} }. 
\end{equation}

 Let $k \in [0, m]$. For fixed $j \geq 1$, if $k \geq j$, then $\frac{ \binom{k}{j} }{ \binom{m}{j} } = \prod_{\ell = 0}^{j-1} \frac{k-\ell}{m - 
 \ell}$. If $m \geq 2j$, then all of the factors of the form $\frac{k-\ell}{m-\ell}$
 are in $[0, 1]$, and, moreover, we find that 
 $\big| \frac{k-\ell}{m-\ell} - \frac{k}{m} \big| = \frac{\ell(m-k)}{m(m-\ell)} \leq \frac{\ell}{m - \ell} \leq \frac{2\ell}{m}$. 
 If $a_i, b_i \in [0, 1]$ for all $i \in \{ 1, 2, \ldots, j \}$, 
 then $\big| \prod_{i=1}^{j} a_{i} - \prod_{i=1}^{j} b_{i} \big| \leq \sum_{i=1}^{j} |a_i - b_i|$, 
 as may be verified via a telescoping argument. Recalling that $k \leq m$, we thus obtain that 
\begin{equation}\label{proddiff}
 \left| \frac{ \binom{k}{j} }{ \binom{m}{j} } - \left( \frac{k}{m} \right)^{j} \right| 
 \leq \frac{j(j-1)}{m}. 
\end{equation}
 If $k < j$, then $\frac{ \binom{k}{j} }{ \binom{m}{j} }$ vanishes. 
 Also if $k < j$, 
 then $0 \leq \left( \frac{k}{m} \right)^{j} \leq \left( \frac{j-1}{m} \right)^{j} = O_{j}\left( \frac{1}{m} \right)$. 
 This gives us that 
\begin{equation}\label{estimatequotient}
 \frac{\binom{k}{j}}{\binom{m}{j}} = \left( \frac{k}{m} \right)^{j} + O_{j}\left( \frac{1}{m} \right) 
\end{equation}
 uniformly for integers $k \in [0, m]$ as $m \to \infty$. From \eqref{definegamma} and 
 \eqref{estimatequotient}, we have that 
\begin{equation}\label{gammasumO}
 \gamma_{k, m} = 
 \sum_{j=0}^{ \deg H } h_{j} 
 \left( \left( \frac{k}{m} \right)^{j} + O_{j}\left( \frac{1}{m} \right) \right), 
\end{equation}
 uniformly for integers $k \in \{ 0, 1, \ldots, m \}$ as $m \to \infty$. 
 Consequently, the relations in \eqref{coefficientsHX} and 
 \eqref{gammasumO} together give us that 
\begin{equation}\label{gammaOH} 
 \gamma_{k, m} = H\left( \frac{k}{m} \right) + O_{H}\left( \frac{1}{m} \right), 
\end{equation}
 and \eqref{gammaOH} holds uniformly for integers $k \in \{ 0, 1, \ldots, m \}$ as $m \to \infty$. 
 Since we have established that $H(x) > 0$ for each $x$ in $[0, 1]$, 
 this and \eqref{gammaOH} together give us that:
 For sufficiently large $m$, the relation $\gamma_{k, m} > 0$ holds for all indices $k$. 

 From \eqref{defineRX} and \eqref{HgammaB} together, we find that 
\begin{equation}\label{fromRandH} 
 R(X) = \sum_{k=0}^{m} \gamma_{k, m} \binom{m}{k} X^{k+r} (1-X)^{m-k+s}, 
\end{equation}
 and we proceed to rewrite the right-hand side of \eqref{fromRandH} with 
 Bernstein basis polynomials, writing
\begin{equation}\label{finalRX} 
 R(X) = \sum_{k=0}^{m} \gamma_{k, m} \frac{ \binom{m}{k} }{ \binom{m+r+s}{k+r} } 
 B_{k+r, m+r+s}(X). 
\end{equation}
 For $m$ sufficiently large, the coefficients of the terms on the right of \eqref{finalRX} are all nonnegative. 
 Since this property holds for $m$ sufficiently large, the same property 
 holds for $n = m+r+s$ sufficiently large, and hence the desired result. 
\end{proof}

 The following Lemma is essentially the converse direction of Theorem 3.6 from the work of Fenner et al.\ \cite{FennerGreenHomer2026}. 

\begin{lemma}\label{lemmaBernstein}
 Let $P \in \mathbb{Z}[X]$ be such that $P(0), P(1) \in \{ 0, 1 \}$ and such that $P(x) \in (0, 1)$ for each $x$ in $(0, 1)$. Then $P$ 
 is in $D(X)$ (cf.\ \cite[Theorem 3.6]{FennerGreenHomer2026}). 
\end{lemma}

\begin{proof}
 Both $P$ and $1-P$ satisfy the conditions of Lemma \ref{lemmalargen}. Choose an integer 
 $N \geq \deg P$ sufficiently large so that the conclusion of Lemma \ref{lemmalargen} holds, 
 at degree $N$, for both $P$ and $1-P$. We may then write
\begin{equation}\label{PXexpandB} 
 P(X) = \sum_{k=0}^{N} \beta_{k, N} B_{k, N}(X), 
\end{equation}
 with $\beta_{k, N} \geq 0$ for each coefficient of the form $\beta_{k, N}$ arising in 
 the expansion in \eqref{PXexpandB}. 
 Since $1 = \sum_{k=0}^{N} B_{k, N}(X)$ (recalling \eqref{displayBernstein}), 
 the coefficients of $1-P$ arising as a linear combination of expressions of the form $B_{k, N}(X)$
 are of the form $1 - \beta_{k, N}$. 
 By our choice of $N$, the Bernstein coefficients $1 - \beta_{k, N}$ of $1-P$ are nonnegative, and hence
\begin{equation}\label{boundbeta} 
 0 \leq \beta_{k, N} \leq 1
\end{equation}
 for $k \in \{ 0, 1, \ldots, N \}$. 
 
 Write $P(X) = \sum_{j=0}^{\deg P} a_{j} X^{j}$, where $a_j \in \mathbb{Z}$.  The binomial theorem gives us that  $ X^j = \sum_{k=j}^{N} 
 \binom{N-j}{k-j} X^k (1-X)^{N-k} $  for all $j \in \{ 0, 1, \ldots, \deg P \}$. This allows us to rewrite $P(X)$ so that 
\begin{equation}\label{PXckN} 
 P(X) = \sum_{k=0}^{N} c_{k, N} X^{k} (1-X)^{N-k} 
\end{equation}
 for $ c_{k, N} = \sum_{j=0}^{\min\{ k, \deg P \}} a_j \binom{N-j}{k-j}$, 
 noting that the coefficients of the form $c_{k, N}$ are all integers. 
 Comparing the coefficients in the expansions in \eqref{PXexpandB} 
 and \eqref{PXckN} allows us to conclude 
 that $ c_{k, N} = \binom{N}{k} \beta_{k, N}$. 
 From \eqref{boundbeta}, we may deduce that 
 $0 \leq c_{k, N} \leq \binom{N}{k}$, 
 and Pinch's lemma, as given as Lemma \ref{Pinchlemma}, 
 allows us to conclude that $P(X) \in D_{N}(X) \subseteq D(X)$.
\end{proof}

 The following result required for our purposes is a special case of a result due to Pinch
 \cite[Proposition 10]{Pinch1985}. 

\begin{lemma}\label{Pinchnotsparse}
 (Pinch, 1985) For $a \in \mathbb{R}$, if there exists an element $P$ in $D(X)$ such that 
 $0 < P(a) < 1$, then $a$ is not sparse \cite[Proposition 10]{Pinch1985}. 
\end{lemma}

 Now, let $K \subseteq \mathbb{R}$ be compact, and define
 $ \| p \|_{K} = \max_{x \in K} | p(x) | $ and 
 $ B(K, \mathbb{Z}) = \{ p \in \mathbb{Z}[X] : \| p \|_{K} < 1 \} $ and 
\begin{equation}\label{defineJKZ}
 J(K, \mathbb{Z}) = \{ x \in K : p(x) = 0 \ \text{for each} \ p \in B(K, \mathbb{Z}) \}. 
\end{equation}
 We also write $J_0(K, \mathbb{Z})$ in place of the union of all of the complete conjugacy classes of algebraic integers
 contained in $K$. 
 For $n \geq 2$, the $n^{\text{th}}$ \emph{diameter} of $K$
 is defined so that 
 $$ d_{n}(K) = \max_{x_1, x_2, \ldots, x_n \in K} \left( \prod_{1 \leq i < j \leq n} 
 \left| x_{i} - x_{j} \right| \right)^{\frac{2}{n(n-1)}}. $$
 The \emph{transfinite diameter} of $K$ is then defined so that 
\begin{equation}\label{definetransfinite} 
 d(K) = \lim_{n \to \infty} d_{n}(K). 
\end{equation}

 We require the following result due to Ferguson \cite[Theorem 6.4]{Ferguson1968}. 

\begin{theorem}\label{thmFerguson}
 (Ferguson, 1968) Let $K$ denote a compact subset of $\mathbb{R}$. 
 If $d(K) < 1$, then 
 $J(K, \mathbb{Z}) = J_0(K, \mathbb{Z})$ \cite[Theorem 6.4]{Ferguson1968}. 
\end{theorem} 

 Let $E \subseteq \mathbb{R}$ be compact. By letting 
\begin{equation}\label{definetau}
 \tau_{n}(E) = \inf\{ \| p \|_{E} : p \in \mathbb{R}[X] \ \text{is monic of degree $n$} \}, 
\end{equation}
 the standard equality between the transfinite diameter and the Chebyshev constant
 \cite{Ferguson1968} gives us that
\begin{equation}\label{rootlimit}
 d(E) = \lim_{n \to \infty} \tau_{n}^{1/n}(E), 
\end{equation}
 again with reference to the work of Ferguson \cite{Ferguson1968}. 
 We also recall that the order-$m$ Chebyshev polynomial of the first kind
 is such that 
\begin{equation}\label{defineChebyshev}
 T_m(\cos \theta) = \cos(m \theta). 
\end{equation}
 For a compact subset $C$ of $\mathbb{R}$, 
 we let $\| \cdot \|_{C}$ denote the uniform/supremum norm
 for $C$, which, for a continuous function 
 $p$ defined on $C$, gives us that 
 $\| p \|_{C} = \max_{x \in C} |p(x)|$. 

\begin{lemma}\label{dKless1}
 If $a \not\in [0, 1]$, then $d\big( [0, 1] \cup \{ a \} \big) < 1$. 
\end{lemma}

\begin{proof}
 For a positive integer $m$, the leading coefficient of the order-$m$ Chebyshev polynomial is $2^{m-1}$. This gives us that 
\begin{equation}\label{defineumX} 
 u_{m}(X) = 2^{1-2m} T_{m}(2X-1)
\end{equation}
 is a monic polynomial of degree $m$. From the relation in \eqref{defineChebyshev}, we find that: If $x \in [0, 1]$, then $|T_{m}(2x-1)| \leq 
 1$. From this relation along with \eqref{defineumX}, we find that 
\begin{equation}\label{normum} 
 \| u_{m} \|_{[0, 1]} \leq 2^{1-2 m}. 
\end{equation}
 Now, for an arbitrary integer $n$ exceeding $1$, 
 let $p_{n}(X) = (X - a) u_{n-1}(X)$, 
 yielding a monic polynomial of degree $n$ that vanishes when evaluated at $a$. 

 Now, define $M_a = \max_{0 \leq x \leq 1} |x-a|$. Also, writing $K$ in place of $[0, 1] \cup \{ a \}$, 
 observe that 
\begin{align*}
 \| p_{n} \|_{K}
 & = \max_{x \in K} | p_n(x) | \\ 
 & = \max\left\{ \max_{x \in [0, 1]} |p_{n}(x)|, |p_{n}(a)| \right\} \\ 
 & = \max\left\{ \| p_{n} \|_{[0, 1]}, 0 \right\} \\ 
 & = \| p_{n} \|_{[0, 1]}. 
\end{align*}
 The definitions of $p_n$ and $M_a$ together
 give us that 
\begin{equation}\label{boundpnx} 
 |p_{n}(x)| = |x - a| |u_{n-1}(x)| \leq M_a \| u_{n-1} \|_{[0, 1]}
\end{equation}
 for each element $x$ in $[0, 1]$. Taking the maximum value of 
 $|p_{n}(x)| $ for $x \in [0, 1]$, we find that \eqref{boundpnx}
 gives us that 
\begin{equation}\label{pn01upper} 
 \| p_{n} \|_{[0, 1]} \leq M_a \| u_{n-1} \|_{[0, 1]}. 
\end{equation}
 A combined application of \eqref{normum}, the relation $\| p_{n} \|_{K} = \| p_{n} \|_{[0, 1]}$, along with the relation in \eqref{pn01upper} 
 gives us that 
\begin{equation}\label{upperpnK}
 \| p_{n} \|_{K} \leq M_a 2^{1-2(n-1)}. 
\end{equation}
 From the definition in \eqref{definetau} together with \eqref{upperpnK} (recalling that $p_n$ is monic
 and of degree $n$), we find that 
\begin{equation}\label{taunKleq} 
 \tau_{n}(K) \leq M_{a} 2^{1 - 2(n-1)}. 
\end{equation}
 From the limiting relation in \eqref{rootlimit} together with the upper bound in 
 \eqref{taunKleq}, we find that 
\begin{equation}\label{dKlimsup} 
 d(K) \leq \limsup_{n \to \infty} \left( M_a 2^{1 - 2(n-1)} \right)^{\frac{1}{n}} = \frac{1}{4}, 
\end{equation}
 and, from \eqref{dKlimsup}, we obtain the desired conclusion whereby $d(K) < 1$. 
 \end{proof}

\begin{lemma}\label{lastlemma}
 Let $K \subseteq \mathbb{R}$ be compact, with $0, 1 \in K$. Assume that $d(K) < 1$ and $J_0(K, \mathbb{Z}) = \{ 0, 1 \}$. Then there is a 
 polynomial $Q \in \mathbb{Z}[X]$
 satisfying $Q(0) = Q(1) = 0$ 
 and satisfying
\begin{equation}\label{Qxin01} 
 Q(x) \in (0, 1)
\end{equation}
 for each element $x$ in $K \setminus \{ 0, 1 \}$. 
\end{lemma}

\begin{proof}
 From the assumptions such that $K \subseteq \mathbb{R}$ 
 is compact, that $d(K) < 1$, and that 
 $J_0(K, \mathbb{Z}) = \{ 0, 1 \}$, the above formulation of Ferguson's theorem \cite[Theorem 6.4]{Ferguson1968}, as in 
 Theorem \ref{thmFerguson} above, gives us that 
\begin{equation}\label{fromFerguson} 
 J(K, \mathbb{Z}) = J_0(K, \mathbb{Z}) = \{ 0, 1 \}. 
\end{equation} 

 For the case whereby $K = \{ 0, 1 \}$, we may let $Q = 0$, so that 
 \eqref{Qxin01} holds vacuously for 
 $x$ in $K \setminus \{ 0, 1 \}$. 

 Now, assume that $K \setminus \{ 0, 1 \} \neq \varnothing$. 
 This leads us to pick a point $\xi$ in the nonempty set difference $K \setminus \{ 0, 1 \}$. 
 From \eqref{fromFerguson}, we see that 
 $\xi \not\in J(K, \mathbb{Z})$. Recalling \eqref{defineJKZ}, 
 since $\xi \not\in J(K, \mathbb{Z})$, we deduce that 
 there exists a polynomial $p \in B(K, \mathbb{Z})$ satisfying $p(\xi) \neq 0$, 
 noting that $p$ is not the zero polynomial. 
 Now, define
\begin{equation}\label{definecalZ}
 \mathcal{Z} = \{ x \in K \setminus \{ 0, 1 \} : p(x) = 0 \}. 
\end{equation}
 Recalling that $p$ is a single-variable and nonzero polynomial,  
 we find that $p$ has only finitely many real zeroes. 
 This gives us that the right-hand side of 
 \eqref{definecalZ} is finite, and we proceed to write 
 $\mathcal{Z} = \{ x_1, x_2, \ldots, x_m \}$, allowing for the 
 possibility that $m = 0$. 

 Again recalling \eqref{defineJKZ}, 
 we find that $x_i \not\in J(K, \mathbb{Z})$ for each index $i$ $ \in$ $ \{ 1$, $2$, $ \ldots$, $ m \}$. 
 Consequently, for each $i$ in $\{ 1, 2, \ldots, m \}$, 
 there exists a polynomial $p_{i} \in B(K, \mathbb{Z})$
 satisfying 
\begin{equation}\label{pixineq0} 
 p_i(x_i) \neq 0. 
\end{equation} 
 Now, let $x \in K \setminus \{ 0, 1 \}$. If $p(x) \neq 0$, then $x$ is not a common zero 
 of the polynomials $p$, $p_1$, $p_2$, $\ldots$, $p_m$. 
 If $p(x) = 0$, 
 then $x \in \mathcal{Z}$, and this gives us that $x = x_i$ for some index $i$ in $\{ 1, 2, \ldots, m \}$. 
 Consequently, the relation in \eqref{pixineq0}
 gives us that $p_i(x) = p_i(x_i) \neq 0$, 
 i.e., so that $x$ is not a common zero among 
 $p$, $p_1$, $p_2$, $\ldots$, $p_m$. 
 Consequently, we obtain the inclusion
\begin{equation}\label{forwardinclude} 
 \{ x \in K : p(x) = p_1(x) = p_2(x) = \cdots = p_m(x) = 0 \} \subseteq \{ 0, 1 \}. 
\end{equation}
 Let $q \in B(K, \mathbb{Z})$. From the definition of $B(K, \mathbb{Z})$, we find that 
 $\| q \|_{K} < 1$. 
 Since $0$ and $1$ are in $K$, we find that 
 $|q(0)| < 1$ and $|q(1)| < 1$. 
 Since $q$ is in $\mathbb{Z}[X]$, we find that $q(0)$ and $q(1)$ are integers, 
 so that $q(0) = q(1) = 0$. Moreover, since $q$ is in $B(K, \mathbb{Z})$, we see that 
 $0$ and $1$ are both in $J(K, \mathbb{Z})$. 

 Applying the above conclusion to the particular polynomials
 $p$, $p_1$, $p_2$, $\ldots$, $p_m$ in $B(K, \mathbb{Z})$, 
 we find that 
\begin{equation}\label{pi0equalspi1} 
 p(0) = p(1) = 0 \quad \text{and} \quad p_i(0) = p_i(1) = 0 
\end{equation}
 for $i \in \{ 1, 2, \ldots, m \}$. 
 Combining this with \eqref{forwardinclude}, we obtain
\begin{equation}\label{mutualinclusion}
 \{ x \in K : p(x) = p_1(x) = p_2(x) = \cdots = p_m(x) = 0 \} = \{ 0, 1 \}. 
 \end{equation}

 Since 
 $p$, $p_1$, $p_2$, $\ldots$, $p_m$ are in $B(K, \mathbb{Z})$,    
 we find that 
\begin{equation}\label{normKpi} 
 \| p \|_{K} < 1 \quad \text{and} \quad \| p_i \|_{K} < 1 
\end{equation}
 for each index $i$ in $\{ 1, 2, \ldots, m \}$. 
 We proceed to define
\begin{equation}\label{definerho} 
 \rho = \max\left\{ \| p \|_{K}, \| p_1 \|_{K}, \| p_2 \|_K, \ldots, \| p_m \|_{K} \right\}. 
\end{equation}
 From \eqref{definerho} and from the preceding norm inequalities, 
 we have that $\rho < 1$. Moreover, since $p(\xi) \neq 0$, we have
 $\| p \|_{K} \geq |p(\xi)| > 0$. 
 Since $\rho \geq \| p \|_{K}$, it follows that 
 $ 0 < \rho < 1$. 
 Now, let $N$ be a sufficiently large natural number
 satisfying
\begin{equation}\label{Nlarge}
 (m + 1) \rho^{2N} < 1. 
\end{equation}
 Also, we define 
\begin{equation}\label{defineQX} 
 Q(X) = p^{2N}(X) + \sum_{i=1}^m p_i^{2N}(X). 
\end{equation}
 Since 
 $p$, $p_1$, $p_2$, $\ldots$, $p_m$ $\in$ $B(K, \mathbb{Z}) \subseteq \mathbb{Z}[X]$, 
 the definition in \eqref{defineQX} gives
 $Q(X) \in \mathbb{Z}[X]$. 
 Moreover, from \eqref{pi0equalspi1}, we find that $Q(0) = Q(1) = 0$. 
 
 Now, let $x$ be an element of $K \setminus \{ 0, 1 \}$. From the equality in \eqref{mutualinclusion},  we deduce that at least one of $p(x)$, 
 $p_1(x)$, $p_2(x)$, $\ldots$, $p_m(x)$
 is nonzero. As a consequence, we obtain that 
 $ Q(x) = |p(x)|^{2N} + \sum_{i=1}^{m} |p_{i}(x)|^{2N} > 0$. Also, we find that 
\begin{align*}
 Q(x) & = |p(x)|^{2N} + \sum_{i=1}^{m} |p_i(x)|^{2N} \\ 
 & \leq \| p \|_{K}^{2N} + \sum_{i=1}^{m} \| p_i \|_{K}^{2N} \\ 
 & \leq (m+1) \rho^{2N} \\
 & < 1, 
\end{align*}
 exploiting \eqref{Nlarge}. 
 This gives us that $0 < Q(x) < 1$ for each element $x$ in $K \setminus \{ 0, 1 \}$, as desired. 
\end{proof}

\subsection{Proof of Pinch's conjecture}

\begin{theorem}
 Let $a$ denote a totally real algebraic integer. If $a$ is sparse, then each algebraic conjugate of $a$
 distinct from $a$ is in $(0, 1)$. 
\end{theorem}

\begin{proof}
 Assume that $a$ is a sparse and totally real algebraic integer. 
 
 If $a$ is an integer, then the minimal polynomial (with respect to an indeterminate $X$) of $a$ over $\mathbb{Q}$
 is $X-a$, so that $a$ does not have any nontrivial algebraic conjugates, 
 and the desired conclusion holds vacuously. 
 So, we henceforth assume that $a$ is not an integer. 

 By way of contradiction, suppose that $a \in (0, 1)$. In this case, the identity polynomial
 $P(X) = X$ would be such that $0 < P(a) = a < 1$, but then 
 Pinch's lemma, as formulated above in Lemma \ref{Pinchnotsparse}, 
 would give us that $a$ would not be sparse, 
 contradicting our above assumption that $a$ is sparse. 
 So, we henceforth let $a$ be outside of $(0, 1)$. 
 Moreover, since we are letting $a$ be a non-integer, 
 we have that 
\begin{equation}\label{notinclosedunit}
 a \not\in [0, 1]. 
\end{equation}

 Now, by way of contradiction, assume that $a$ has a nontrivial conjugate $a^{\ast}$ such that 
\begin{equation}\label{starnotinopen} 
 a^{\ast} \not\in (0, 1). 
\end{equation}
 From the assumption that $a$ is totally real, we find that $a^{\ast} \in \mathbb{R}$. We see that $a^{\ast}$ could not be equal to $0$, 
 because, otherwise, letting $m_a(X)$ denote the minimal polynomial of $a$, 
 we would have that $m_a(0) = 0$, 
 with $X$ dividing $m_a(X)$, 
 but the irreducibility of $m_a$ over $\mathbb{Q}$ 
 would thus imply that $m_a(X) = X$, 
 i.e., so that $a = 0$, but this contradicts \eqref{notinclosedunit} (and contradicts that $a \not\in \mathbb{Z}$). 
 A symmetric argument gives us that $a^{\ast} \neq 1$, i.e., since, otherwise, 
 we would have that $m_a(X) = X-1$, with $a = 1$. 
 So, from \eqref{starnotinopen}, we obtain that 
\begin{equation}\label{aastnotinunit} 
 a^{\ast} \not\in [0, 1]. 
\end{equation}

 We proceed to define $K = [0, 1] \cup \{ a \}$. From \eqref{notinclosedunit}, we find that 
 Lemma \ref{dKless1} gives us that 
\begin{equation}\label{strongerpossible}
 d(K) < 1,
\end{equation} 
 noting that a stronger estimate may be obtained from the proof of Lemma \ref{dKless1}. 

 We see that 
\begin{equation}\label{01containedJ0} 
 \{ 0, 1 \} \subseteq J_0(K, \mathbb{Z}), 
\end{equation}
 since the conjugacy classes for $0$ and $1$ are $\{ 0 \}$ and $\{ 1 \}$, respectively, 
 and each of these singleton sets is contained in $K$. 

 Now, let $\alpha \in J_0(K, \mathbb{Z})$. We thus have that the conjugacy class of $\alpha$ is contained in $K$. 
 
 By way of contradiction, suppose that $a$ is a conjugate of $\alpha$. 
 So the minimal polynomial $m_{\alpha}(X)$ would have $a$ as a root. 
 Consequently, the minimal polynomial $m_a(X)$ would divide $m_{\alpha}(X)$
 in $\mathbb{Q}[X]$, and, since both of these polynomials are monic and irreducible, 
 we would have that $m_{\alpha}(X) = m_{a}(X)$. So, the conjugacy classes of $\alpha$ and $a$ would be equal, 
 i.e., so that the conjugacy class of $\alpha$ would necessarily contain $a^{\ast}$, 
 recalling that $a^{\ast}$ is defined as a nontrivial conjugate of $a$. 
 By definition of a nontrivial conjugate, we have that $a^{\ast} \neq a$, and this and 
 \eqref{aastnotinunit} together give us that 
 $a^{\ast} \not\in K = [0, 1] \cup \{ a \}$. 
 We thus obtain a contradiction, since the conjugacy class of $\alpha$ is contained in $K$. 

 So, none of the conjugates of $\alpha$ equal $a$. Since all of these conjugates are in $K = [0, 1] \cup \{ a \}$, 
 all of the conjugates of $\alpha$ are in $[0, 1]$. 

 Let the conjugates of $\alpha$ be denoted with $\beta_1$, $\beta_2$, $\ldots$, $\beta_d$
 for some positive integer $d$. 
 The preceding argument gives us that 
\begin{equation}\label{boundbetai} 
 0 \leq \beta_i \leq 1
\end{equation}
 for each index $i$ in $\{ 1, 2, \ldots, d \}$. If at least one expression of the form $\beta_i$ were to be equal to $0$, then the minimal 
 polynomial $m_{\alpha}(X)$ would be the identity polynomial, 
 with $m_{\alpha}(X) = X$, and this would give us that $\alpha = 0$. 
 Alternatively, suppose that each expression of the form $\beta_i$ were positive. 
 The bounds in \eqref{boundbetai} 
 would give us that 
\begin{equation}\label{boundprodbeta} 
 0 < \prod_{i=1}^{d} \beta_i \leq 1. 
\end{equation}
 From the assumption that $\alpha$ is an algebraic integer, we have that its minimal polynomial is monic and has coefficients in 
 $\mathbb{Z}$, so that 
\begin{equation}\label{minpoly0} 
 \prod_{i=1}^{d} \beta_i = (-1)^d m_{\alpha}(0) \in \mathbb{Z}. 
\end{equation}
 Since the left-hand side of the equality \eqref{minpoly0} is positive, i.e., from \eqref{boundprodbeta}, the upper bound in 
 \eqref{boundprodbeta}, together with the left-hand side of the equation in \eqref{minpoly0} being an integer,  gives    us that 
 $ \prod_{i = 1}^{d} \beta_i = 1$. 
 Each expression of the form $\beta_i$ is in $(0, 1]$, from 
 \eqref{boundbetai} and 
 \eqref{boundprodbeta} together. 
 We see that a finite product of values in $(0, 1]$ equals $1$ only when each factor equals $1$, 
 i.e., so that $\beta_1 = \beta_2 = \cdots = \beta_d = 1$, which implies that $\alpha = 1$. 
 This establishes the reverse inclusion corresponding to \eqref{01containedJ0}, writing 
\begin{equation}\label{J0equals2set}
 J_0(K, \mathbb{Z}) = \{ 0, 1 \}. 
\end{equation}

 So, from \eqref{strongerpossible} and \eqref{J0equals2set} together, we find that all of the conditions given in Lemma \ref{lastlemma} 
 hold. So, Lemma \ref{lastlemma} gives us that 
 there is a polynomial $Q(X)$ in $\mathbb{Z}[X]$
 satisfying $Q(0) = Q(1) = 0$
 and satisfying 
\begin{equation}\label{universal}
 \forall x \in K \setminus \{ 0, 1 \} \ 0 < Q(x) < 1. 
\end{equation}
 Noting the containment of $(0, 1)$ in $K \setminus \{ 0, 1 \}$,   we deduce that $Q(x) \in (0, 1)$ for each member $x$ in $(0, 1)$.    From  
  Lemma \ref{lemmaBernstein}, we thus have that 
\begin{equation}\label{QinDX} 
 Q \in D(X).
\end{equation} 
 Moreover, since $a \in K \setminus \{ 0, 1 \}$ (recalling that $K = [0, 1] \cup \{ a \}$ and recalling \eqref{notinclosedunit}),  we find that 
 \eqref{universal} gives us that  $0 < Q(a) < 1$. From this and \eqref{QinDX} together, we find that 
 Pinch's lemma, as in Lemma \ref{Pinchnotsparse} above, 
 gives us that $a$ is not sparse, contradicting our initial assumption that $a$ is sparse. 
 We have thus arrived at a contradiction giving us that it is not the case that
 $a$ has a nontrivial conjugate $a^{\ast}$ satisfying $a^{\ast} \not\in (0, 1)$. 
\end{proof}

\subsection*{Acknowledgements}
 The author acknowledges extensive interactions with GPT-5.6 Pro during the exploratory and proof-development stages of this work. All 
 AI-generated suggestions were substantially revised, corrected, and independently verified by the author, who assumes full responsibility 
 for the mathematical content. 

\bibliographystyle{plain}
\bibliography{seprefe}

@article {Pinch1985,
    AUTHOR = {Pinch, R. G. E.},
     TITLE = {{$a$}-convexity},
   JOURNAL = {Math. Proc. Cambridge Philos. Soc.},
  FJOURNAL = {Mathematical Proceedings of the Cambridge Philosophical
              Society},
    VOLUME = {97},
      YEAR = {1985},
    NUMBER = {1},
     PAGES = {63--68},
      ISSN = {0305-0041,1469-8064},
   MRCLASS = {11R45 (11R04 52A01)},
  MRNUMBER = {764493},
MRREVIEWER = {P.\ McMullen},
       DOI = {10.1017/S0305004100062587},
       URL = {https://doi.org/10.1017/S0305004100062587},
}

@article {FennerGreenHomer2026,
    AUTHOR = {Fenner, Stephen and Green, Frederic and Homer, Steven},
     TITLE = {Fixed-parameter extrapolation and aperiodic order},
   JOURNAL = {Discrete Comput. Geom.},
  FJOURNAL = {Discrete \& Computational Geometry. An International Journal
              of Mathematics and Computer Science},
    VOLUME = {76},
      YEAR = {2026},
    NUMBER = {1},
     PAGES = {1--74},
      ISSN = {0179-5376,1432-0444},
   MRCLASS = {52C23 (52-02)},
  MRNUMBER = {5084666},
       DOI = {10.1007/s00454-025-00816-4},
       URL = {https://doi.org/10.1007/s00454-025-00816-4},
}

@article {MasakovaPelantovaSvobodova2000,
    AUTHOR = {Mas\'akov\'a, Zuzana and Pelantov\'a, Edita and Svobodov\'a,
              Milena},
     TITLE = {Characterization of cut-and-project sets using a binary
              operation},
   JOURNAL = {Lett. Math. Phys.},
  FJOURNAL = {Letters in Mathematical Physics},
    VOLUME = {54},
      YEAR = {2000},
    NUMBER = {1},
     PAGES = {1--10},
      ISSN = {0377-9017,1573-0530},
   MRCLASS = {52C23 (82D25)},
  MRNUMBER = {1846718},
       DOI = {10.1023/A:1007684402406},
       URL = {https://doi.org/10.1023/A:1007684402406},
}

@article{BhattacharyaRosenfeld2000,
title = {a-Convexity},
journal = {Pattern Recognition Letters},
volume = {21},
number = {10},
pages = {955-957},
year = {2000},
issn = {0167-8655},
doi = {https://doi.org/10.1016/S0167-8655(00)00051-9},
url = {https://www.sciencedirect.com/science/article/pii/S0167865500000519},
author = {Prabir Bhattacharya and Azriel Rosenfeld},
}

@article{MasakovaPateraPelantova2001,
author = {Masáková, Z and Patera, J and Pelantová, E},
title = {Exceptional algebraic properties of the three quadratic irrationalities observed in quasicrystals},
journal = {Canadian Journal of Physics},
volume = {79},
number = {2-3},
pages = {687-696},
year = {2001},
doi = {10.1139/p01-003},

URL = { 
    
        https://doi.org/10.1139/p01-003},}

@article{Svobodova2001,
author = {Milena Svobodová},
title = {S-convexity and cut-and-project sets},
journal = {Ferroelectrics},
volume = {250},
number = {1},
pages = {175--177},
year = {2001},
publisher = {Taylor \& Francis},
doi = {10.1080/00150190108225059},
URL = { 
            https://doi.org/10.1080/00150190108225059
    },}

@article{Ferguson1968,
 author = {Ferguson, Le Baron O.},
 title = {Uniform approximation by polynomials with integral coefficients I},
 fjournal = {Pacific Journal of Mathematics},
 journal = {Pac. J. Math.},
 volume = {27},
 pages = {53--69},
 year = {1968},
 language = {English},
 doi = {10.2140/pjm.1968.27.53},
URL = { 
            https://doi.org/10.2140/pjm.1968.27.53},}

\end{document}